\documentclass[reqno]{amsart}

\usepackage{enumerate}
\usepackage{comment}
\usepackage{empheq}
\usepackage{amsmath}
\usepackage{amssymb}
\usepackage{array}
\usepackage{graphicx}
\usepackage{tikz}
\usepackage{pgf}
\usepackage{pgfpages}
\usepackage{comment}
\usetikzlibrary{decorations.pathreplacing}
\usetikzlibrary{calc}
\usetikzlibrary{decorations}

\usetikzlibrary{matrix}
\tikzstyle{vertex}=[circle,thin,draw=black!100,fill=black!100, inner sep=0pt, minimum width=4pt]
\tikzstyle{vertexsm}=[circle,thin,draw=black!100,fill=black!100, inner sep=0pt, minimum width=1.2pt]
\tikzstyle{vertexnm}=[circle,thin,draw=black!100,fill=white!100, inner sep=0pt, minimum width=4pt]
\tikzstyle{vertexg}=[circle, draw=black, inner sep=1pt, style=densely dotted, minimum width=4pt]
\tikzstyle{vertexinf}=[circle,thin,draw=black!100,fill=white!100, inner sep=0pt, minimum width=10pt]
\tikzstyle{pedge}=[draw=black!100,-]
\tikzstyle{nedge}=[draw=black!100,densely dashed]
\tikzstyle{gedge}=[draw=black!100,densely dotted]

\newtheorem{theorem}{Theorem}

\newtheorem{lemma}[theorem]{Lemma}
\newtheorem{corollary}[theorem]{Corollary}

\begin{document}

\title{An extension of Hoffman and Smith's subdivision theorem}
\author{Lee Gumbrell}
\address{Department of Mathematics, Royal Holloway, University of
London, Egham Hill, Egham, Surrey, TW20 0EX, England, UK.}
\email{Lee.Gumbrell.2009@rhul.ac.uk}


\begin{abstract}
In $1975$ Hoffman and Smith showed that for a graph $G\ne\tilde{D}_n$ with an internal path, the value of the largest eigenvalue decreases strictly each time we subdivide the internal path. In this paper we extend this result to show that for a graph $G\ne K_{1,4}$ with a vertex of degree $4$ or more, we can subdivide said vertex to create an internal path and the value of the largest eigenvalue also strictly decreases.
\end{abstract}

\maketitle

In \cite{HoSm1975} (Proposition 2.4) Hoffman and Smith proved an important result about the largest eigenvalue, or spectral radius, of (the adjacency matrix of) a graph, which we shall call $\rho\left(G\right)$ or $\rho$ if it is clear which graph $G$ we are talking about. Define an internal path of a graph to be a sequence of at least two adjacent vertices all with vertex-degrees $2$, except the two end-vertices who each have degree strictly greater than $2$ (and may possibly be the same vertex). Also define a subdivision of an internal path to be the same path, but with one more degree $2$ vertex in the sequence. A subdivided graph $G'$ is isomorphic $G$ except with one more vertex and one more edge on the subdivided internal path. The result they proved is the following:

\begin{theorem}[Hoffman and Smith's subdivision theorem, \cite{HoSm1975}, also see Theorem 3.2.3 of \cite{CRS1997}]\label{T:hoff}
Let $G$ be a graph with an internal path, and let $G'$ be the graph obtained by subdividing an edge on that path. If $G$ is not equal to the graph $\tilde{D}_n$ in Figure \ref{F:dn} then $\rho\left(G'\right)<\rho\left(G\right)$.
\end{theorem}

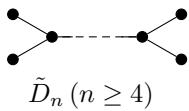
\begin{figure}[h]
\begin{center}
\begin{tabular}{c}
\begin{tikzpicture}[scale=0.6, auto] 
\foreach \pos/\name in
{{(0,0)/a},{(0.866,0.5)/b},{(0.866,-0.5)/c},{(-2,0)/aa},{(-2.866,0.5)/ab},{(-2.866,-0.5)/ac}}
\node[vertex] (\name) at \pos {};
\foreach \pos/\name in
{{(-0.667,0)/ax},{(-1.333,0)/ay}}
\node[] (\name) at \pos {};
\node[] at (-3.066,-0.5) {};
\foreach \edgetype/\source/ \dest in {pedge/a/b,pedge/c/a,pedge/a/ax,pedge/ay/aa,pedge/aa/ab,pedge/aa/ac,nedge/a/aa}
\path[\edgetype] (\source) -- (\dest);
\end{tikzpicture}
\\
$\tilde{D}_n \left(n \geq 4\right)$
\end{tabular}
\caption{The graph $\tilde{D}_n$ on $n+1$ vertices.}
\label{F:dn}
\end{center}
\end{figure}

We will now extend their result. Instead of requiring that the path has at least one edge, we start with a graph $G\ne K_{1,4}$, the star on $5$ vertices, with a vertex $v$ of degree at least $4$ and split this vertex into two new adjacent vertices, each adjacent to at least two of $v$'s neighbours. We will show that this new graph also has largest eigenvalue strictly less than the original graph. Figure \ref{F:hoffsmithsimple} shows the relevant vertices and edges of a graph to demonstrate this process, where $G_v$ is our subdivided graph. Informally speaking, Hoffman and Smith proved the result for paths with number of edges from $1$ to infinity and in Theorem \ref{T:hoffsmith0} we will prove the result for paths of length $0$, completing the picture.

\begin{figure}[h]
\begin{center}
\begin{tabular}{cc}
\begin{tikzpicture}[scale=0.6, auto] 
\draw (-1.5,1.5) -- (1.5,1.5);
\draw (1.5,1.5) -- (1.5,-1.5);
\draw (1.5,-1.5) -- (-1.5,-1.5);
\draw (-1.5,-1.5) -- (-1.5,1.5);
\foreach \pos/\name in
{{(0,0)/v},{(0.707,0.707)/a},{(0.707,-0.707)/b},{(-0.707,-0.707)/c},{(-0.707,0.707)/d}}
\node[vertex] (\name) at \pos {};
\foreach \pos/\name in
{{(0.924,0.383)/aa},{(1,0)/ab},{(0.924,-0.383)/ac},{(-0.924,0.383)/da},{(-1,0)/db},{(-0.924,-0.383)/dc}}
\node[vertexsm] (\name) at \pos {};
\foreach \edgetype/\source/ \dest in {pedge/v/b,pedge/c/v,pedge/a/v,pedge/v/d}
\path[\edgetype] (\source) -- (\dest);
\node[] at (0,0.5) {$v$};
\foreach \pos/\name in
{{(-1.7,0)/zl},{(1.7,0)/zr},{(0,1.7)/zt},{(0,-1.7)/zb}}
\node[] (\name) at \pos {};
\end{tikzpicture}
&
\begin{tikzpicture}[scale=0.6, auto] 
\draw (-1.5,1.5) -- (2.5,1.5);
\draw (2.5,1.5) -- (2.5,-1.5);
\draw (2.5,-1.5) -- (-1.5,-1.5);
\draw (-1.5,-1.5) -- (-1.5,1.5);
\foreach \pos/\name in
{{(0,0)/v},{(1,0)/w},{(1.707,0.707)/a},{(1.707,-0.707)/b},{(-0.707,-0.707)/c},{(-0.707,0.707)/d}}
\node[vertex] (\name) at \pos {};
\foreach \pos/\name in
{{(1.924,0.383)/aa},{(2,0)/ab},{(1.924,-0.383)/ac},{(-0.924,0.383)/da},{(-1,0)/db},{(-0.924,-0.383)/dc}}
\node[vertexsm] (\name) at \pos {};
\foreach \edgetype/\source/ \dest in {pedge/w/b,pedge/c/v,pedge/a/w,pedge/v/d,pedge/v/w}
\path[\edgetype] (\source) -- (\dest);
\node[] at (0,0.5) {$v_1$};
\node[] at (1,0.5) {$v_2$};
\foreach \pos/\name in
{{(-1.7,0)/zl},{(2.7,0)/zr},{(0,1.7)/zt},{(0,-1.7)/zb}}
\node[] (\name) at \pos {};
\end{tikzpicture}
\\
$G$ & $G_v$
\end{tabular}
\caption{An example of subdividing a vertex $v$ of degree at least $4$ from $G$ to create a graph $G_v$ with an internal path. Only the relevant vertices and edges of the graphs have been drawn.}
\label{F:hoffsmithsimple}
\end{center}
\end{figure}
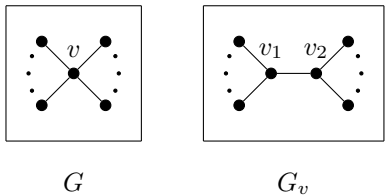

In \cite{Si1987} (also see Theorem 3.2.1 of \cite{CRS1997}), Simi{\'c} showed a similar result about splitting a vertex, but where the two new vertices were not adjacent to each other. The difference between the two results is subtle, and the direct proofs use the same technique, but after proving Theorem \ref{T:hoffsmith0} we will show how it can be used to offer a new proof of Simi{\'c}'s theorem.

We require the following results from Perron and Frobenius, and will refer to them frequently throughout this paper.

\begin{lemma}[Perron-Frobenius; see \cite{GoRo2001}, Theorem 8.8.1]\label{L:pf}
Let $A\left(G\right)$ be the adjacency matrix of a connected graph $G$, then:
\begin{enumerate}[(i)]
	\item $\rho\left(G\right)$ is a simple eigenvalue of $A\left(G\right)$ and, if $z$ is an eigenvector for $\rho$, then none of the entries of $z$ are zero and they all have the same sign;
	\item for another graph $H$, if $A\left(G\right)-A\left(H\right)$ is non-negative then $\rho\left(H\right)\leq\rho\left(G\right)$, with equality if and only if $G$ and $H$ are isomorphic.
\end{enumerate}
\end{lemma}

To begin we require the following simple lemma.

\begin{lemma} \label{L:rhogreater2}
Let $G$ be a connected graph with a vertex $v$ such that the degree $d\left(v\right)\geq4$. Then $\rho\left(G\right)\geq2$, and $\rho\left(G\right)=2$ if and only if $G=K_{1,4}$.
\end{lemma}

\begin{proof}
Consider the vertex $v$ with $d\left(v\right)\geq4$ and four of its adjacent vertices. These vertices will form a graph isomorphic to the star $K_{1,4}$ or $K_{1,4}$ with some extra edges, $K_{1,4}\cup E$ say, where $E$ is the set of extra edges. If the graph $G$ is $K_{1,4}$ then the spectral radius is $2$ and if $K_{1,4}$ is a proper induced subgraph of $G$ then $\rho\left(G\right)>2$ by Perron-Frobenius (Lemma \ref{L:pf}). 

If $G$ is $K_{1,4}\cup E$ then we note that $A(K_{1,4}\cup E)-A(K_{1,4})$ is non-negative so by Perron-Frobenius we have that $\rho(K_{1,4}\cup E)>\rho(K_{1,4})=2$. Finally, if $K_{1,4}\cup E$ is a proper induced subgraph of $G$ then again we get $\rho\left(G\right)>2$ by Perron-Frobenius.
\end{proof}

We are now ready to prove our main result.

\begin{theorem} \label{T:hoffsmith0}
Let $G\ne K_{1,4}$ be a graph with a vertex $v$ such that $d\left(v\right)\geq4$. Expand the vertex $v$ into two adjacent vertices $v_1$ and $v_2$ such that $d\left(v_1\right)\geq2$ and $d\left(v_2\right)\geq2$. Call this graph $G_v$, then $\rho\left(G_v\right)<\rho\left(G\right)$.
\end{theorem}

\begin{proof}
Let $A\left(G\right)z=\rho\left(G\right)z$ with $z>0$ (that is, $z$ is the eigenvector for the largest eigenvalue $\rho$ of $G$). Let the vertices adjacent to $v$ in $G$ be $x_1,\ldots,x_s,y_1,\ldots,y_t$ such that in $G_v$ the vertices $x_1,\ldots,x_s$ are adjacent to $v_1$ and $y_1,\ldots,y_t$ are adjacent to $v_2$.
Let $\hat{z}$ be a vector of length $\left|V\left(G_v\right)\right|$ where the coordinates corresponding to the vertices of $G_v\setminus\left\{v_1,v_2\right\}$ are the same as those in $z$ corresponding to the vertices in $G\setminus\left\{v\right\}$. We will choose what $\hat{z}_{v_1}$ and $\hat{z}_{v_2}$ are and consider the parts of the vector $A\left(G_v\right)\hat{z}$ that are affected by these. We find that

$A\left(G_v\right)\hat{z}=$
\begin{displaymath}
\left(
\begin{array}{c|ccc|cc|ccc|c}
 &  &  &  &  &  &  &  & & \\ \hline
 &  &  &  & 1 & 0 &  &  & & \\
 &  &  &  & \vdots & \vdots &  &  & & \\
 &  &  &  & 1 & 0 &  &  & & \\ \hline
$\ $ & 1 & \cdots & 1 & 0 & 1 & 0 & \cdots & 0 & $\ $ \\
 & 0 & \cdots & 0 & 1 & 0 & 1 & \cdots & 1 & \\ \hline
 &  &  &  & 0 & 1 &  &  & & \\
 &  &  &  & \vdots & \vdots &  &  & &  \\
 &  &  &  & 0 & 1 &  &  & & \\ \hline
 &  &  &  &  &  &  &  &  & \\
\end{array}
\right)
\left(
\begin{array}{c}
 \\ \hline
z_{x_1} \\
\vdots \\
z_{x_s} \\ \hline
\hat{z}_{v_1} \\
\hat{z}_{v_2} \\ \hline
z_{y_1} \\
\vdots \\
z_{y_t} \\ \hline
 \\
\end{array}
\right)
=
\left(
\begin{array}{c}
 \\ \hline
\rho z_{x_1} - z_v + \hat{z}_{v_1} \\
\vdots \\
\rho z_{x_s} - z_v + \hat{z}_{v_1} \\ \hline
\sum z_{x_i} + \hat{z}_{v_2} \\
\sum z_{y_i} + \hat{z}_{v_1} \\ \hline
\rho z_{y_1} - z_v + \hat{z}_{v_2} \\
\vdots \\
\rho z_{y_t} - z_v + \hat{z}_{v_2} \\ \hline
 \\
\end{array}
\right)
\begin{array}{c}
\\
\\
(a) \\
\\
(b) \\
(c) \\
\\
(d) \\
\\
\\
\end{array}
\end{displaymath}
where we have grouped the equations in the resulting vector as (a), (b), (c) and (d).

There are four cases to consider:
\begin{enumerate}
	\item  $z_v\geq\sum z_{x_i}$ and $z_v\geq\sum z_{y_i}$
	\item  $z_v\geq\sum z_{x_i}$ and $z_v < \sum z_{y_i}$
	\item  $z_v < \sum z_{x_i}$ and $z_v\geq\sum z_{y_i}$
	\item  $z_v < \sum z_{x_i}$ and $z_v < \sum z_{y_i}$.
\end{enumerate}
In each case we will show that $A\left(G_v\right)\hat{z}<\rho\left(G\right)\hat{z}$, which then gives $\rho\left(G_v\right)<\rho\left(G\right)$ by Perron-Frobenius.

Case 1. Set $\hat{z}_{v_1}=\hat{z}_{v_2}=z_v$. Equations (a) and (d) then become $\rho z_{x_i}$ and $\rho z_{y_i}$ respectively. Since $\sum z_{x_i} \leq z_v$ and $\sum z_{y_i} \leq z_v$ equations (b) and (c) give $\sum z_{x_i} + z_v \leq 2z_v < \rho z_v$ and $\sum z_{y_i} + z_v \leq 2z_v < \rho z_v$ since Lemma \ref{L:rhogreater2} gives $\rho > 2$. These strict inequalities give the result.

Case 2. Set $\hat{z}_{v_1}=\sum z_{x_i}$ and $\hat{z}_{v_2}=z_v$. Equation (a) becomes $\rho z_{x_i} - z_v + \hat{z}_{v_1}$ for each $i=1,\ldots,s$ but $\sum z_{x_i} \leq z_v$ so $\rho z_{x_i} - z_v + \hat{z}_{v_1} \leq \rho z_{x_i} - z_v + z_v = \rho z_{x_i}$. Equation (c) becomes $\sum z_{x_i} + \sum z_{y_i} = \rho z_v = \rho \hat{z}_{v_2}$ and equation (d) is equal to $\rho z_{y_i}$ for $i=1,\ldots,t$. Our strict inequality comes from equation (b) but in different places depending on a few subcases. Now, $\rho z_{x_i} = z_v + \Sigma_i$ for each $i=1,\ldots,s$ where $\Sigma_i \geq 0$. Therefore $z_v \leq \rho z_{x_i}$ for each $i$ and $s z_v \leq \rho \sum z_{x_i}$. Since $s\geq 2$ equation (b) becomes $\sum z_{x_i} + z_v \leq 2 z_v \leq s z_v \leq \rho \sum z_{x_i} = \rho \hat{z}_{v_1}$.
If $d\left(x_i\right)\geq2$ for at least one $i$ then for that $i$ we have $\Sigma_i > 0$ so $z_v < \rho z_{x_i}$ and $s z_v < \rho \sum z_{x_i}$ giving us the strict inequality. If $d\left(x_i\right)=1$ for all $i$ and $s>2$ then $2 z_v < s z_v$ and again we are done. If $d\left(x_i\right)=1$ and $s=2$ then $z_{x_1}=z_{x_2}$ so $2 z_{x_1} = \sum z_{x_i} \leq z_v = \rho z_{x_i}$. However, $\rho>2$ so $2z_{x_1}<z_v$ and then $\sum z_{x_i} + z_v < 2 z_v = \rho \sum z_{x_i} = \rho \hat{z}_{v_1}$.

Case 3. Set $\hat{z}_{v_1}=z_v$ and $\hat{z}_{v_2}=\sum z_{y_i}$. The proof then follows as in Case 2.

Case 4. Set $\hat{z}_{v_1}=\hat{z}_{v_2}=z_v$. Equations (a) and (d) again become $\rho z_{x_i}$ and $\rho z_{y_i}$ respectively. Since $z_v < \sum z_{x_i}$ and $z_v < \sum z_{y_i}$ equations (b) and (c) give
\begin{eqnarray*}
\sum z_{x_i} + z_v & < \sum z_{x_i} + \sum z_{y_i} = \rho z_v, \\
\sum z_{y_i} + z_v & < \sum z_{y_i} + \sum z_{x_i} = \rho z_v.
\end{eqnarray*}
Again these strict inequalities give the result in this case.
\end{proof}

As a consequence of Theorem \ref{T:hoffsmith0} we can provide an alternative proof of the result of Simi\'{c} in \cite{Si1987} about splitting vertices.

\begin{corollary}[see \cite{Si1987}]
Let $G$ be a graph with a vertex $v$, and let $W_1 \cup W_2$ be a non-trivial bipartition of the vertices adjacent to $v$. Let the graph $G'$ be formed by taking $G \backslash v$ and including two new non-adjacent vertices $v_1$ and $v_2$, where $v_i$ is adjacent to all of the vertices in $W_i$ ($i=1,2$). Then $\rho(G')<\rho(G)$.
\end{corollary}

\begin{proof}
To prove this we shall consider three cases: $d\left(v\right)=2$, $d\left(v\right)=3$ or $d\left(v\right)\geq4$ (note that the non-trivial bipartition of the vertices adjacent to $v$ excludes the possibility that $d\left(v\right)=1$; in that case the largest eigenvalue stays the same).

When $d\left(v\right)=2$, consider whether $v$ is on an internal path or not. If it is, we can subdivide an edge between $v$ and one of its neighbours twice. After doing this we can readily spot a vertex which upon removal induces the graph $G'$. Subdivision and Perron-Frobenius (Theorem \ref{T:hoff} and Lemma \ref{L:pf}, respectively) then tell us that the largest eigenvalue has strictly decreased. If $v$ is not on an internal path it is either on a pendent path, or on a cycle, $C_n$. In the former case $G'$ consists of two disconnected graphs, both subgraphs of $G$, so Perron-Frobenius gives us the result. The largest eigenvalue of $C_n$ is $2$ for all $n$ (as the eigenvalues are $2\cos\left(2\pi j/n\right)$ for $j=0,\ldots,n-1$) and clearly $P_n$ is a subgraph of $C_{n+1}$ so must have a strictly smaller largest eigenvalue.

When $d\left(v\right)=3$, we simply use Perron-Frobenius if $G'$ is disconnected or subdivision twice followed by Perron-Frobenius if not. Finally, if $d\left(v\right)\geq4$, then we simply compare the adjacency matrix of $G'$ with the same graph found using Theorem \ref{T:hoffsmith0} and the result follows by Perron-Frobenius.
\end{proof}

Another application of Theorem \ref{T:hoffsmith0} has been in the search for trivial Salem graphs, and a detailed description of this can be found in the author's thesis (see Chapter 4 of \cite{Gu2013T}). We shall also mention a generalisation of our main result in Theorem \ref{T:expandkn} below. An alternative way of thinking about the subdivision of the vertex $v$ in Theorem \ref{T:hoffsmith0} is that we have expanded $v$ into a $K_2$, the complete graph on two vertices, and shared the neighbours around the two new vertices. A natural question to ask is will the same result hold if we expand the vertex into a complete graph of any size? The answer is yes, provided the vertex is of degree at least $n^2$ and the graph is not isomorphic to $K_{1,n^2}$, although we will not give a full proof here; the method is the same as the proof of Theorem \ref{T:hoffsmith0} and complete details can also be found in the author's thesis (\cite{Gu2013T}). 

Figure \ref{F:hoffsmithcomplex} should help clarify the idea of expanding a vertex, again only showing the relevant vertices and edges of $G$, and a formal description is found below in the statement of Theorem \ref{T:expandkn}.

\begin{figure}[h]
\begin{center}
\begin{tabular}{cc}
\begin{tikzpicture}[scale=0.6, auto] 
\draw (-4,4) -- (4,4);
\draw (4,4) -- (4,-2.5);
\draw (4,-2.5) -- (-4,-2.5);
\draw (-4,-2.5) -- (-4,4);
\draw (-1,0) arc (0:360:0.75cm);
\draw (0,1) arc (-90:270:0.75cm);
\draw (1,0) arc (180:-180:0.75cm);
\node[] at (0,1.75) {$W_1$};
\node[] at (1.75,0) {$W_2$};
\node[] at (-1.75,0) {$W_n$};
\foreach \pos/\name in
{{(0,0)/v}}
\node[vertex] (\name) at \pos {};
\foreach \pos/\name in
{{(-0.85,0)/aa},{(-0.6,0)/ab},{(-0.35,0)/ac},{(0,0.85)/ba},{(0,0.6)/bb},{(0,0.35)/bc},{(0.85,0)/da},{(0.6,0)/db},{(0.35,0)/dc},{(0,-1.5)/fa},{(-1,-1)/fb},{(1,-1)/fc}}
\node[vertexsm] (\name) at \pos {};
\foreach \pos/\name in
{{(-1.3,-0.4)/ga},{(-1.3,0.4)/gb},
{(-0.4,1.3)/ha},{(0.4,1.3)/hb},
{(1.3,-0.4)/ia},{(1.3,0.4)/ib}}
\node[] (\name) at \pos {};
\foreach \edgetype/\source/ \dest in {pedge/v/ga,pedge/v/gb,pedge/v/ha,pedge/v/hb,pedge/v/ia,pedge/v/ib}
\path[\edgetype] (\source) -- (\dest);
\node[] at (0,-0.5) {$v$};
\foreach \pos/\name in
{{(-4.5,0)/zl},{(4.5,0)/zr},{(0,4.5)/zt},{(0,-2.75)/zb}}
\node[] (\name) at \pos {};
\end{tikzpicture}
&
\begin{tikzpicture}[scale=0.6, auto] 
\draw (-4,4) -- (4,4);
\draw (4,4) -- (4,-2.5);
\draw (4,-2.5) -- (-4,-2.5);
\draw (-4,-2.5) -- (-4,4);
\draw (1,0) arc (0:360:1cm);
\node[] at (0,0) {$K_n$};
\draw (-2,0) arc (0:360:0.75cm);
\draw (0,2) arc (-90:270:0.75cm);
\draw (2,0) arc (180:-180:0.75cm);
\node[] at (0,2.75) {$W_1$};
\node[] at (2.75,0) {$W_2$};
\node[] at (-2.75,0) {$W_n$};
\foreach \pos/\name in
{{(-1,0)/va},{(0,1)/vb},{(1,0)/vc}}
\node[vertex] (\name) at \pos {};
\foreach \pos/\name in
{{(-1.85,0)/aa},{(-1.6,0)/ab},{(-1.35,0)/ac},{(0,1.85)/ba},{(0,1.6)/bb},{(0,1.35)/bc},{(1.85,0)/da},{(1.6,0)/db},{(1.35,0)/dc},{(0,-2)/fa},{(-1.414,-1.414)/fb},{(1.414,-1.414)/fc}}
\node[vertexsm] (\name) at \pos {};
\foreach \pos/\name in
{{(-2.3,-0.4)/ga},{(-2.3,0.4)/gb},
{(-0.4,2.3)/ha},{(0.4,2.3)/hb},
{(2.3,-0.4)/ia},{(2.3,0.4)/ib}}
\node[] (\name) at \pos {};
\foreach \edgetype/\source/ \dest in {pedge/va/ga,pedge/va/gb,pedge/vb/ha,pedge/vb/hb,pedge/vc/ia,pedge/vc/ib}
\path[\edgetype] (\source) -- (\dest);
\node[] at (0.6,1.2) {$v_1$};
\node[] at (1.35,-0.55) {$v_2$};
\node[] at (-1.2,-0.55) {$v_n$};
\foreach \pos/\name in
{{(-4.5,0)/zl},{(4.5,0)/zr},{(0,4.5)/zt},{(0,-2.75)/zb}}
\node[] (\name) at \pos {};
\end{tikzpicture}
\\
$G$ & $G_{K_n}$
\end{tabular}
\caption{An example of expanding a vertex $v$ (from $G$, with $d\left(v\right)\geq4$) to a $K_n$ to create the graph $G_{K_n}$. Only the relevant vertices and edges of the graphs have been drawn and the $W_i$ are partitions of the neighbours of $v$, each containing at least $n$ of them.}
\label{F:hoffsmithcomplex}
\end{center}
\end{figure}
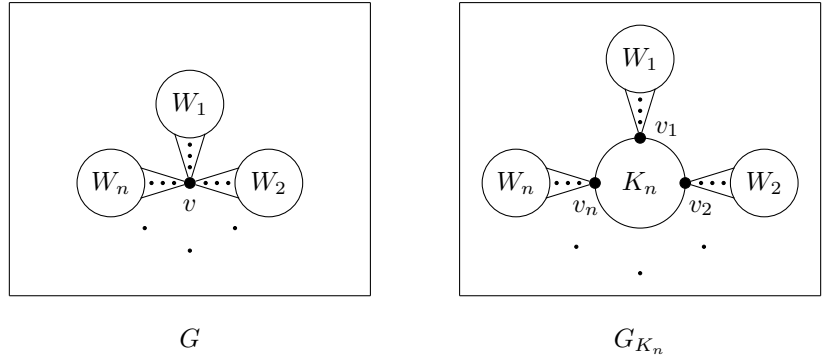

\begin{theorem} \label{T:expandkn}
Let $G\ne K_{1,n^2}$ be a graph with a vertex $v$ such that $d\left(v\right)\geq n^2$. Group the vertices of $G$ adjacent to $v$ into $n$ partitions $W_1,\ldots,W_n$, where each $W_i$ contains at least $n$ vertices. Let the vertices of $K_n$ be labeled $v_1,\ldots,v_n$ and let $G_{K_n}$ be the graph isomorphic to $K_n \cup G\backslash v$ along with edges joining the vertices of $W_i$ of $v_i$ (for $i=1,\ldots,n$). Then $\rho\left(G_{K_n}\right)<\rho\left(G\right)$.
\end{theorem}

\begin{proof}
A full proof can be found in \cite{Gu2013T} but can also be seen as a generalisation of the proof of Theorem \ref{T:hoffsmith0}.
\end{proof}

\end{document}